\documentclass[12pt]{amsart}
\usepackage{amsmath}
\usepackage{amsfonts}
\usepackage{amssymb}
\usepackage{verbatim}
\usepackage{hyperref}
\usepackage{color}

\newcommand{\R}{\mathbb R}
\newcommand{\C}{\mathbb C}
\newcommand{\Z}{\mathbb Z}
\newcommand{\Hb}{\mathbb H}
\newcommand{\T}{\mathbb T}

\newtheorem{theorem}{Theorem}[section]
\newtheorem{corollary}[theorem]{Corollary}
\newtheorem{lemma}[theorem]{Lemma}

\newtheorem{proposition}[theorem]{Proposition}
\theoremstyle{definition}
\newtheorem{definition}[theorem]{Definitions}
\theoremstyle{remark}
\newtheorem{rem}[theorem]{Remark}

\numberwithin{equation}{section}

\makeatletter
\renewcommand*{\eqref}[1]{%
  \hyperref[{#1}]{\textup{\tagform@{\ref*{#1}}}}%
}
\makeatother

\title[Hermite Multipliers on Modulation Spaces]
{Hermite multipliers on Modulation Spaces}      

\author[D. G. Bhimani]{Divyang G. Bhimani}
\address[D. G. Bhimani]{Centre for Applicable Mathematics (CAM)\\
Tata Institute of Fundamental Research\\
560012 Bangalore, India}
\email{divyang@tifrbng.res.in}

\author[R. Balhara]{Rakesh Balhara}
\address[R. Balhara]{Department of Mathematics\\
 Indian Institute of Science\\
560 012 Bangalore, India}
\email{rakeshbalhara@gmail.com}

\author[S. Thangavelu]{Sundaram Thangavelu}
\address[S. Thangavelu]{Department of Mathematics\\
 Indian Institute of Science\\
560 012 Bangalore, India}
\email{veluma@math.iisc.ernet.in}


\begin{document}
\begin{abstract}
We study multipliers associated to the Hermite operator $H=-\Delta + |x|^2$  on modulation spaces $M^{p,q}(\mathbb R^d)$. We prove that the operator $m(H)$ is bounded on $M^{p,q}(\mathbb R^d)$ under standard  conditions on $m,$ for suitable choice of $p$ and $q$.  As an application, we point out  that the solutions to the free  wave  and Schr\"odinger equations associated to $H$  with initial data in a modulation space will remain in the same modulation space for all times.  We also point out that Riesz transforms associated to $H$ are bounded on some modulation spaces.
\end{abstract} 
\maketitle

\section{Introduction}
The main aim of this article is to study the boundedness properties of Hermite multipliers on modulation spaces. We quickly recall the setup in order to state our results and we refer to Section \ref{p} for details.  The spectral decomposition of the Hermite operator $ H = -\Delta +|x|^2 $ on $ \mathbb R^d $ is given by $ H = \sum_{k=0}^\infty (2k+d) P_k $ where $ P_k $ stands for the orthogonal projectionof $ L^2(\mathbb R^d) $ onto the eigenspace corresponding to the eigenvalue $(2k+d)$. Given a bounded function $m$ defined on the set of all natural numbers, we define the operator $ m(H) $ simply by setting
$ m(H) = \sum_{k=0}^\infty m(2k+d) P_k.$  We say that $ m $ is an $ L^p $ multiplier for the Hermite expansions if $ m(H) $ extends to $ L^p $ as a bounded operator. Sufficient conditions on $ m $ are known so that $ m $ is an $L^p $ multiplier, see e.g.\cite{MRS,MS,HWJC,HW}. In this article we are interested in multipliers $ m $ which define bounded operators on the modulation spaces $ M^{p,q}(\mathbb R^d) $.

Recall that a tempered distribution $ f $ on $ \R^d $ belongs to the modulation space $ M^{q,p}(\mathbb R^d)$ if the Fourier-Wigner transform of $ f $ and the Gaussian $ \Phi_0(\xi) = \pi^{-d/2} e^{-\frac{1}{2}|\xi|^2} $ defined by 
$$ \langle \pi(x+iy)f,\Phi_{0}\rangle  = \int_{\mathbb R^d} e^{i(x \cdot \xi+\frac{1}{2}x\cdot y)} f(\xi+y)\Phi_0(\xi) d\xi $$
belongs to the mixed norm space $ L^p(\mathbb R^d, L^q(\mathbb R^d))$ consisting of functions $F(x,y)$ for which the norms 
$$ \|F\|_{L^{p,q}}  = \left( \int_{\mathbb R^d} \|F(x, \cdot)\|_q^p dx \right)^{1/p} $$ are finite (See Section \ref{MW} below). These spaces have several interesting properties not shared by the $ L^p $ spaces. For example, $M^{p,p}(\mathbb R^d)$ are invariant under the Fourier transform and  $ M^{q,1}(\mathbb R^d)$  are algebras under pointwise multiplication.

For the multiplier operators  $m(D)$, the problem of establishing sufficient conditions on $m$ that make the operator $m(D)$ bounded on  $L^{p}$ has a long history. As it appears often in various applications, like solving linear dispersive  PDE, e.g., wave/Schr\"odinger equations, 
for more detail, we refer to \cite{WB, rsw, gro2, MS} and the reference therein.  It is well known that the   operator (See Definition \ref{fme} below) with Fourier multiplier $e^{i|\xi|^{\alpha}} (\alpha>2)$ is bounded on $L^{p}(\mathbb R^d)$ if and only if  $p=2.$

The study of  Fourier multiplier operators, which are of the form $ m( -\Delta) $ in the context of modulation space $M^{2,1}(\mathbb R^d)$ was initiated in the works of  Wang-Zhao-Guo \cite{WZG}.  
In fact, in the consequent year  B\'enyi-Gr\"{o}chenig-Okoudjou-Rogers \cite{gro2}   have shown  that  the Fourier multiplier operator with multiplier  $e^{i|\xi|^{\alpha}}  (\alpha \in [0,2])$  is bounded on  $M^{p,q}(\mathbb R^d)$ for all  $1\leq p,q \leq \infty.$ The cases $\alpha =1$ and $\alpha =2$ are particularly interesting and have been studied intensively in PDE, because they occur in the time evolution of the wave equation ($\alpha =1$) and the free Schr\"odinger operator $(\alpha =2).$  Thus, the Sch\"odinger and wave propagators are not  $L^p (p\neq 2)$-bounded but  $M^{p,q}$- bounded for all $1\leq p,q \leq \infty.$  In fact, this  leads to  fixed-time estimates for Sch\"odinger and wave propagators and some of their applications to well-posedness results on modulation spaces $M^{p,q}(\mathbb R^d).$   Modulation spaces  have turned out to be very fruitful  in numerous applications in various problems in analysis and PDE.  And yet there has been  a lot of ongoing interest in these spaces  from the harmonic analysis and PDE points of view.  We refer to the recent survey  \cite{rsw} and the references therein.

Coming back to the Hermite operator,  we note that  Thangavelu \cite{ST1}  (See also \cite[Theorem 4.2.1]{ST})  has  proved an analogue of the  H\"ormander-Mikhlin type multiplier theorem for Hermite expansions on $L^p(\mathbb R^d).$ Specifically, he showed that under certain conditions on $m$, the operator $m(H)$ is bounded on $L^{p}(\mathbb R^d) (1<p< \infty).$ It is well known that the harmonic oscillator  $H=-\Delta + |x|^2$ appears in  various applications. We refer to the recent article \cite{cj} and the reference therein for details.

Taking all these considerations into  account,  we are motivated  to study Hermite multipliers  $m(H)$  on modulation spaces $M^{p,q}(\mathbb R^d).$ The conditions we impose on the multiplier $ m $ is the standard one in terms of local Sobolev spaces. We assume that $ m $ is defined on the whole of $ \mathbb R .$
Let $0\neq \psi \in C_{0}^{\infty} (\mathbb R^{+})$ be a fixed cut-off function with support contained in the interval $[\frac{1}{2}, 1],$ and define the scale invariant localized Sobolev norm of order $\beta$ of $m\in L^{\infty}(\mathbb R^+)$ by 
$$\|m\|_{L^{2}_{\beta, sloc}}= \sup_{t>0} \|\psi m(t \cdot)\|_{L^{2}_{\beta}}.$$
We then have the following theorem.

\begin{theorem}\label{FT} 
 Let $1< p\leq q  \leq 2$ or   $2\leq q \leq p < \infty.$
  Suppose that $\|m\|_{L^{2}_{\beta, sloc }} < \infty$ for some  $\beta > (2d+1)/2.$  Then the operator $m(H)$ is bounded on $M^{q,p}(\mathbb R^d)$.
\end{theorem}

We remark that this theorem is not sharp. For $ m $ to be an $ L^p $ multiplier it is sufficient to assume the condition on $ m $ with $ \beta > d/2.$ We believe the same is true in the case of multipliers on modulation spaces though our method of proof requires a stronger assumption on $ m.$  However,  it is worth noting that $M^{p,p}(\mathbb R^d)$ for $p>2$ is a much wider class than $L^{p}(\mathbb R^d)$ (See Lemma \ref{pl}\eqref{inclu} below). Thus,  Theorem \ref{FT}  shows that the  H\"ormander-Mikhlin multiplier  type theorem is true for a much wider class than $L^p(\mathbb  R^d).$

We deduce Theorem \ref{FT} from a corresponding result on the polarised Heisenberg  group $ \Hb^d_{pol} = \mathbb R^d \times \mathbb R^d \times \mathbb R $  which is of Euclidean dimension $ (2d+1).$ Let $ \mathcal { \tilde{L}} $ stand for the sublaplacian on $  \Hb^d_{pol} $ and define $ m(\mathcal {\tilde{L}}) $ using spectral theorem (See Subsection \ref{HP} below for precise definitions).  We then have the following transference result.

\begin{theorem}\label{MT}
 Let $1< p\leq q  \leq 2$ or   $2\leq q \leq p < \infty.$ Then $ m(H) $ is bounded on $ M^{q,p}(\mathbb R^d) $ whenever $ m(\mathcal { \tilde{ L}}) $ is bounded on $ L^p(\Hb^d_{pol}).$
\end{theorem}

The above theorem allows us to deduce some interesting corollaries for Hermite multipliers. For example,   let $ R_j = (-\frac{\partial}{\partial \xi_j}+\xi_j) H^{-1/2} , j = 1,2, ...,d $ be the Riesz transforms associated to $ H.$  It is well known that these Riesz transforms are bounded on $ L^p(\mathbb R^d), 1 < p < \infty $ but their behavior on modulation spaces have not been studied. By considering Riesz transforms on $ G^d $ we obtain

\begin{corollary}
 Let $1< p\leq q  \leq 2$ or   $2\leq q \leq p < \infty.$ Then the Riesz transforms $ R_j $ are bounded on $ M^{q,p}(\mathbb R^d).$
\end{corollary}

Another interesting corollary is the following result about solutions of the wave equation associated to $ H.$  Consider the following Cauchy problem:

$$\partial_{t}^{2}u(x,t)= -Hu(x,t),~~~ u(x,0)=0,~~~ \partial_{t}u(x,0)=f(x)$$
whose solution is given by $ u(x,t) = H^{-1/2} \sin(tH^{1/2})f(x).$

\begin{corollary}\label{AW} Let $ u $ be the solution of the above Cauchy problem for $ H.$ Then for $1< p\leq q  \leq 2$ or   $2\leq q \leq p < \infty $ we have the estimate $ \|u(\cdot,t)\|_{M^{q,p}} \leq C_t \|f\|_{M^{q,p}}$ provided $ |\frac{1}{p}-\frac{1}{2}| < \frac{1}{2d}.$
\end{corollary}

When $ p = q $ Theorem \ref{MT} as well as Corollary \ref{AW} can be improved. This will be achieved using transference as before, but now the transference is from multiplier theorems for multiple Fourier series. Given a function $ m $ on $ \R $ we define a Fourier multiplier $ T_m $ on $ L^p(\T^d)$ where $ \T^d $ is the $d-$dimensional torus by
$$ T_mf(x) =  \sum_{\mu \in \Z^d} m(|\mu|) \hat{f}(\mu) e^{i\mu \cdot x}.$$
Here $ \hat{f}(\mu) $ are the Fourier coefficients of $ f $ and $|\mu| = \sum_{j=1}^d |\mu_j|.$ Using the connection (See Proposition \ref{tpt} below) between Fourier multipliers $ T_m $ on $ L^p(\T^d) $ an¡d $ m(H) $ on $ M^{p,p} $ we prove
\begin{theorem}\label{cstn} 
 Let $1 <  p  < \infty.$
  Suppose that $\|m\|_{L^{2}_{\beta, sloc }} < \infty$ for some $\beta > d/2.$  Then the operator $m(H)$ is bounded on $M^{p,p}(\mathbb R^d)$.
\end{theorem}

We also have the following improvement of Corollary \ref{AW}. More generally, we consider multipliers of the form

\begin{eqnarray}\label{st}
m(2k+d)=  \frac{e^{i(2k+d)^{\gamma}}}{(2k+d)^{\beta}}, \   \  (\beta>0, \gamma>0).
\end{eqnarray}
\begin{theorem}\label{HMT}
Let $1\leq p < \infty,$ and  $|\frac{1}{p}-\frac{1}{2}|<\frac{\beta}{d \gamma}$, and let $m$ be given by \eqref{st}. Then  $m(H) $ is bounded on  $ M^{p,p}(\mathbb R^d).$ In particular,   Corollary \ref{AW} is valid on the bigger range $ |\frac{1}{p}-\frac{1}{2}| < \frac{1}{d} $ when $ p =q.$
\end{theorem}
\noindent
We now turn our attention to a multiplier which occurs in the time evolution of the free Schr\"odinger equation associated to $H.$ Specifically, we consider the Cauchy problem for the Schr\"odinger equation associated to $H:$
$$i\partial_{t}u(x,t) -Hu(x,t)=0, \   u(x,0)= f(x)$$
whose solution is given by $ u(x,t) = e^{itH}f(x).$

\begin{theorem}\label{mso}
The Schr\"odinger  propagator  $m(H)=e^{itH}$ is bounded on $M^{p,p}(\mathbb R^d)$ for all $1\leq p < \infty.$
\end{theorem}
\noindent

Recently Cordero-Nicola \cite[Section 5.1]{en} have studied the operator $m(H)=e^{itH}$ on Wiener amalgam spaces (closely related to modulation spaces). Later Kato-Kobayashi-Ito \cite{kki} have given a refinement of Cordero-Nicola's results in the  context of Wiener amalgam spaces. We have studied (Theorem \ref{mso}) the boundeness of $m(H)=e^{itH}$ in the context of modulation spaces. Here we would like to point out that our method of proof is completely different from the method  used in the context of Wiener amalgam spaces. We also believe that our method of proof is much simpler than the proofs available  in the literature. Our  proof relies on properties of Hermite and special Hermite functions and illustrates  the importance of these functions in the study of such problems.

Finally, we note that  modulation spaces have been used as regularity classes for  initial data associated to Cauchy problems for nonlinear dispersive equations (eg., NLS, NLW, etc..) but so far mainly for the nonlinear dispersive equations  associated to the   Laplacian without potential  ($D=\Delta$) \cite{rsw, WZG, WB}. There is also an ongoing interest to use harmonic analysis tools (specifically, multiplier results) to solve modern nonlinear PDE problems (See e.g., \cite{cj}). Thus, we strongly believe that our results  will be useful in the future for studying   nonlinear dispersive  equations associated to  $H$  in the realm of modulation spaces.

The paper is organized as follows.  In Section \ref{p}, we introduce  notations and preliminaries which will be used  in the sequel. Specifically, in Subsection \ref{HP}, we  introduce the Heisenberg  and polarised  Heisenberg groups,  the sublaplacians corresponding to these groups, and spectral multipliers associated to these sublaplacians. In  Subsection \ref{tml}, we prove the transference result which connects spectral multiplier on polarised Heisenberg groups and reduced polarised  Heisenberg groups.  In Subsection \ref{MW}, we introduce modulation  spaces and  recall some of their basic properties. In Section \ref{hmvt}, we prove Theorems \ref{FT} and \ref{MT}. In Section \ref{hmft}, we prove Theorems  \ref{cstn} and  \ref{HMT}. In Section \ref{hmsos}, we prove Theorem \ref{mso}.

\section{Notation and Preliminaries}\label{p}
\noindent
\subsection{Notations} The notation $A \lesssim B $ means $A \leq cB$ for  some constant $c > 0 $. The symbol $A_{1}\hookrightarrow A_{2}$ denotes the continuous embedding  of the topological linear space $A_{1}$ into $A_{2}.$ 
If $\alpha= (\alpha_1,..., \alpha_d)\in \mathbb N^d$ is a multi-index, we set $|\alpha|= \sum_{j=1}^d \alpha_j, \alpha != \prod_{j=1}^d\alpha_j!.$ If $z=(z_1,...,z_d)\in \mathbb C^d,$ we put $z^{\alpha}=\prod_{j=1}^dz_{j}^{\alpha_j}$.
We denote the $d-$dimensional torus by $\mathbb T^d \equiv [0, 2\pi)^d,$ and the $L^{p}(\mathbb T^d)-$norm  by
$$\|f\|_{L^{p}(\mathbb T^d)}=\left( \int_{[0, 2\pi)^d} |f(t)|^p dt \right)^{1/p}.$$
The class of  trigonometric polynomials on  $\mathbb T^d$ is denoted by $\mathcal{P}(\mathbb T^d).$
 The mixed  $L^{p}(\mathbb R^{d}, L^{q}(\mathbb R^{d}))$ norm is denoted by 
 $$\|f\|_{L^{p,q}}=\left( \int_{\mathbb R^{d}} \left ( \int_{\mathbb R^{d}} |f(x,y)|^{q} dy \right)^{p/q} dx\right )^{1/p} \ \  (1\leq p, q < \infty),$$
the $L^{\infty}(\mathbb R^{d})$ norm  is $\|f\|_{L^{\infty}}= \text{ess.sup}_{ x\in \mathbb R^{d}}|f(x)| $.  The Schwartz class is denoted  by $\mathcal{S}(\mathbb R^{d})$ (with its usual topology), and the space of tempered distributions is denoted by  $\mathcal{S'}(\mathbb R^{d}).$  For $x=(x_1,\cdots, x_d), y=(y_1,\cdots, y_d) \in \mathbb R^d, $ we put $x\cdot y = \sum_{i=1}^{d} x_i y_i.$ Let $\mathcal{F}:\mathcal{S}(\mathbb R^{d})\to \mathcal{S}(\mathbb R^{d})$ be the Fourier transform  defined by  
\begin{eqnarray*}
\mathcal{F}f(\xi)=\widehat{f}(\xi)=  (2\pi)^{-d/2} \int_{\mathbb R^{d}} f(x) e^{- i  x \cdot \xi}dx, \  \xi \in \mathbb R^d.
\end{eqnarray*}
Then $\mathcal{F}$ is a bijection  and the inverse Fourier transform  is given by
\begin{eqnarray*}
\mathcal{F}^{-1}f(x)=f^{\vee}(x)= (2\pi)^{-d/2}\int_{\mathbb R^{d}} f(\xi)\, e^{ i x\cdot \xi} d\xi,~~x\in \mathbb R^{d}.
\end{eqnarray*}
It is well known that the  Fourier transform can be uniquely extended to $\mathcal{F}:\mathcal{S}'(\mathbb R^d) \to \mathcal{S}'(\mathbb R^d).$

\subsection{Heisenberg group and Fourier multipliers}\label{HP}
We consider the Heisenberg group $\mathbb H^{d}=\mathbb C^{d} \times \mathbb R$ with the group law
$$(z,t)(w,s)=(z+w, t+s+\frac{1}{2}  \text{Im} (z\cdot \bar{w})).$$
Sometimes we use real coordinates $ (x,y,t) $ instead of $ (z,t) $ if $ z = x+iy, x,y \in \R^d.$
In order to define the (group) Fourier transform on the Heisenberg group we briefly recall the following family of irreducible unitary representations. For each non-zero real $\lambda,$ we have a representation $ \pi_\lambda $ realised on $ L^2(\R^d) $ as follows:
$$ \pi_\lambda(z,t)\varphi(\xi) = e^{i\lambda t} e^{i(x \cdot \xi+\frac{1}{2}x \cdot y)} \varphi(\xi+y)  $$ where $ \varphi \in L^2(\R^d).$
For $ f \in L^1(\Hb^d),$ one defines the operator
$$ \hat{f}(\lambda) = \int_{\Hb^d} f(z,t) \pi_\lambda(z,t) dz dt.$$
The operator valued function $ f \rightarrow \hat{f}(\lambda) $ is called the group Fourier transform of $ f $ on $ \Hb^d.$ We refer to \cite{ST2} for more about the group Fourier transform.

The Fourier transform initially defined on $ L^1(\Hb^d)\cap L^2(\Hb^d) $ can be extended to the whole of $ L^2(\Hb^d)$ and we have a version of Plancherel theorem. Specifically, when $f\in L^1\cap L^{2}(\mathbb H^d),$ it can be shown that $\hat{f}(\lambda)$ is a Hilbert-Schmidt operator and the Plancherel theorem holds:
$$ \int_{\mathbb H^d} |f(z,t)|^2 dz dt = \frac{2^{d-1}}{\pi^{d+1}} \int_{-\infty}^{\infty}  \|\hat{f}(\lambda) \|_{HS}^2 |\lambda|^d d\lambda$$
where $\|\cdot\|_{HS}$ is the Hilbert-Schmidt norm given by $\|T\|_{HS}^{2}= {tr} (T^*T),$ for $T$ a bounded operator,  $T^*$ being the adjoint operator of $T.$
Under the assumption that $ \hat{f}(\lambda) $ is of trace class we have the inversion formula
$$ f(z,t) = (2\pi)^{-d-1} \int_{-\infty}^\infty tr(\pi_\lambda(z,t)^* \hat{f}(\lambda)) |\lambda|^d d\lambda.$$ 
Let $ f^\lambda $ stand for the inverse Fourier transform of $ f $ in the \textit{central variable} $t$
\begin{equation}
\label{eq:inverseFT}
f^\lambda(z) = \int_{-\infty}^\infty f(z,t) e^{i\lambda t} dt.
\end{equation}
By taking the Euclidean Fourier transform of $ f^\lambda(z)$  in the variable $\lambda$, we obtain
\begin{equation}
\label{eq:lambdaFT}
f(z,t)=\frac{1}{2\pi}\int_{-\infty}^{\infty}e^{-i\lambda t}f^{\lambda}(z)\,d\lambda.
\end{equation}
Recalling the definition of $ \pi_\lambda $ we see that we can write the group Fourier transform as $$ \hat{f}(\lambda) = \pi_\lambda(f^\lambda) $$ where  
$ \pi_\lambda(f^\lambda) = \int_{\mathbb C^d} f^\lambda(z) \pi_\lambda(z,0) dz.$ The operator which takes  a function $ g $ on $ \C^d $ into the operator
$$ \int_{\C^d} g(z) \pi_\lambda(z,0) dz $$ is called the Weyl transform of $ g $ and is denoted by $ W_\lambda(g)$. Thus $ \widehat{f}(\lambda) = W_\lambda(f^\lambda)$.
With these notations we can rewrite the inversion formula as
$$ f(z,t) = (2\pi)^{-d-1} \int_{-\infty}^\infty  e^{-i\lambda t} tr(\pi_\lambda(z,0)^*\pi_\lambda(f^\lambda)) |\lambda|^d d\lambda.$$

On the Heisenberg group, we have the  vector fields

$$T= \frac{\partial}{\partial t},~~~X_{j}= \frac{\partial}{\partial x_{j}} +\frac{1}{2} y_{j} \frac{\partial}{\partial t}, ~~~Y_{j}= \frac{\partial}{\partial y_j}-\frac{1}{2} x_{j}  \frac{\partial}{\partial t}, \ \ j =1,..., d,$$ and they form a basis for the Lie algebra of left invariant vector fields on the Heisenberg group.  The  second-order operator
$$\mathcal{L}= - \sum_{j=1}^{d} (X_{j}^2+ Y_{j}^2)$$
is called the  sublaplacian which  is self-adjoint  and nonnegative  and hence admits a spectral decomposition
$$\mathcal{L}= \int_{0}^{\infty} \lambda  \ dE_{\lambda}.$$
Given a bounded function $m$  defined on $(0, \infty)$ one can define the operator $m(\mathcal{L})$ formally by setting 
$$m(\mathcal{L})f= \int_{0}^{\infty} m(\lambda) \  dE_{\lambda}f.$$
It can be shown that $ \widehat{m(\mathcal{L})f}(\lambda) = \hat{f}(\lambda)m(H(\lambda)) $ where $ H(\lambda) = -\Delta+|\lambda^2 |x|^2 $ are the scaled Hermite operators. Hence, the operators $ m(\mathcal{L})f $ are examples of Fourier multipliers on the Heisenberg group. More generally, (right) Fourier multipliers on the Heisenberg group are operators defined by $ \widehat{T_Mf}(\lambda) = \hat{f}(\lambda) M(\lambda) $ where $ M(\lambda),$  called the multiplier is a family of bounded linear operators on $ L^2(\R^d).$

The boundedness properties of these operators have been studied by several authors, see e.g. \cite{HW, MS, HWJC}. We make use of the following result.
Let $0\neq \psi \in C_{0}^{\infty} (\mathbb R^{+})$ be a fixed cut-off function with support contained in the interval $[\frac{1}{2}, 1],$ and define the scale invariant localized Sobolev norm of order $\beta$ of $m$ by 
$$\|m\|_{L^{2}_{\beta, sloc}}= \sup_{t>0} \|\psi m(t \cdot)\|_{L^{2}_{\beta}}.$$

\begin{theorem}[M\"uller-Stein, Hebisch \cite{MS,HWJC}]\label{MSH} If $\|m\|_{L^{2}_{\beta, sloc }} < \infty$ for some $\beta > (2d+1)/2,$ then $m(\mathcal{L})$ is bounded on  $L^{p}(\mathbb H^{d})$ for $1<p<\infty,$ and of weak type $(1,1).$ 
\end{theorem}

\noindent
We make use of this theorem in proving our main result. Actually we need an analogue of the above result in the context of polarised Heisenberg group
$ \Hb^d_{pol} $ which is just $ \R^d \times \R^d \times \R $ with the group law
$$ (x,y,t)(x',y',t') = (x+x',y+y', t+t'+x'\cdot y).$$

\noindent
A basis for the algebra of  left invariant vectors fields on this group are given by
$$T= \frac{\partial}{\partial t},~~ \tilde{ X}_{j}= \frac{\partial}{\partial x_{j}} + y_{j} \frac{\partial}{\partial t}, ~~~\tilde{Y}_{j}= \frac{\partial}{\partial y_j}, \ \ j =1,..., d.$$ The sublaplacian is then defined as the second-order operator
$$\mathcal{\tilde{L} }= - \sum_{j=1}^{d} (\tilde{X}_{j}^2+ \tilde{Y}_{j}^2).$$
The group $  \Hb^d_{pol} $ is isomorphic to $ \Hb^d $ and the isomorphism is given by the map $ \Phi : \Hb^d \rightarrow \Hb^d_{pol},~~
\Phi(x,y,t) = (x,y, t+ \frac{1}{2}  x \cdot y). $ Note that $ \Phi $ is measure preserving and it is easy to check that
$$  \mathcal L (f \circ \Phi^{-1}) = \tilde{\mathcal L}f \circ \Phi$$ for reasonable functions $ f $ on $ \Hb^d_{pol}.$ In view of this, an analogue of Theorem \ref{MSH} is true for $ m(\tilde{\mathcal L}) .$ This also follows from the fact that the above theorem is valid in a more general context of H-type groups, see \cite{MRS, HWJC, HW}.

Using the isomorphism $ \Phi $ we can define the following family of representations for $ \Hb^d_{pol} .$ For each non zero real $ \lambda $, the representations  $ \rho_\lambda = \pi_\lambda \circ \Phi^{-1} $ are irreducible and unitary. We can use them to define Fourier transform on the group $ \Hb^d_{pol} .$

\subsection{Hermite and special Hermite functions} The spectral decomposition of $H=-\Delta + |x|^2$ is given by  the Hermite expansion. Let $\Phi_{\alpha}(x),  \  \alpha \in \mathbb N^d$ be the normalized Hermite functions which are products of one dimensional Hermite functions. More precisely, 
$ \Phi_\alpha(x) = \Pi_{j=1}^d  h_{\alpha_j}(x_j) $ 
where 
$$ h_k(x) = (\sqrt{\pi}2^k k!)^{-1/2} (-1)^k e^{\frac{1}{2}x^2}  \frac{d^k}{dx^k} e^{-x^2}.$$ The Hermite functions $ \Phi_\alpha $ are eigenfunctions of $H$ with eigenvalues  $(2|\alpha| + d)$  where $|\alpha |= \alpha_{1}+ ...+ \alpha_d.$ Moreover, they form an orthonormal basis for $ L^2(\R^d).$ The spectral decomposition of $ H $ is then written as
$$ H = \sum_{k=0}^\infty (2k+d) P_k,~~~~ P_kf(x) = \sum_{|\alpha|=k} \langle f,\Phi_\alpha\rangle \Phi_\alpha$$ 
where $\langle\cdot, \cdot \rangle $ is the inner product in $L^2(\mathbb{R}^d)$. Given a function $m$ defined and bounded  on the set of all natural numbers we can use the spectral theorem to define $m(H).$ The action of $m(H)$ on a function $f$ is given by 
$$m(H)f= \sum_{\alpha \in \mathbb N^d} m(2|\alpha| +d) \langle f, \Phi_{\alpha} \rangle \Phi_{\alpha} = \sum_{k=0}^\infty m(2k+d)P_kf.$$
This operator  $m(H)$ is bounded on $L^{2}(\mathbb R^d).$ This follows immediately from the Plancherel theorem for the Hermite expansions as  $m$ is bounded. 

On the other hand, the mere boundedness of $m$ is not sufficient  to imply  the $L^{p}$ boundedness of $m(H)$ for $p\neq 2.$ 
So, we need to impose  some conditions $m$ to ensure that $m(H)$ is bounded on $L^{p}(\mathbb R^d).$   The boundedness results of $m(H)$ on $L^{p}(\mathbb R^d)$ have been studied by several authors,  see e.g. \cite{ST1, ST}. Our aim in this paper is to study the boundedness of $ m(H) $ on modulation spaces.

In the sequel, we make use of some properties of special Hermite functions $ \Phi_{\alpha, \beta} $ which are defined as follows. Let $ \pi(z) = \pi_1(z,t) $ and define
\begin{eqnarray}\label{SHF}
\Phi_{\alpha,\beta}(z) =  (2\pi)^{-d/2} \langle \pi(z)\Phi_\alpha,\Phi_\beta\rangle.
\end{eqnarray}
Then it is well known that these so called special Hermite functions form an orthonormal basis for $ L^2(\mathbb C^d).$ Moreover, they are eigenfunctions of the special Hermite operator $ L $ which is defined by the equation $$ \mathcal L(f(z)e^{it}) = e^{it} Lf(z).$$ Indeed, we  can show (\cite[Theorem 1.3.3]{ST}) that $$ L\Phi_{\alpha,\beta} =(2|\alpha|+d)\Phi_{\alpha, \beta}.$$
For a function  $f\in L^2(\mathbb C^d)$ we have the eigenfunction expansion
\begin{eqnarray*}
f(z)= \sum_{\alpha \in \mathbb N^d} \sum_{\beta \in \mathbb N^d} \langle f, \Phi_{\alpha, \beta} \rangle \Phi_{\alpha, \beta} (z)
\end{eqnarray*}
which is called the special Hermite expansion.

Let $f$ and $g$ be two measurable functions on  $\mathbb C^d.$ We recall that the $\lambda-$twisted convolution ($ 0\neq \lambda \in \mathbb R$) of  $f$ and $g$ is the function $f\ast_{\lambda} g$ defined by 
$$f\ast_{\lambda} g (z)= \int_{\mathbb C^d}f(z-w) g(w)e^{i \frac{\lambda}{2} \text{Im} (z\cdot \bar{w})} dw$$
for all $z$ such that the integral exits. When $\lambda =1,$ we simply call them twisted convolution and denote them by $f\times g.$ It is well known that the twisted convolution of special Hermite functions satisfies \cite[Proposition 1.3.2]{ST} the following relation
\begin{eqnarray}\label{ptc}
\Phi_{\alpha, \beta} \times \Phi_{\mu, v} =  (2\pi)^{d/2} \delta_{\beta, \mu} \Phi_{\alpha, v}
\end{eqnarray}
where $\delta_{\beta, \mu}=1$ if $\beta= \mu$, otherwise it is 0.
Using  the identity \eqref{ptc}, we can show that the  special Hermite expansions can be written (\cite[Section 2.1]{ST}) in the compact form as follows:
\begin{eqnarray*}
f(z) & = &  (2\pi)^{-d/2} \sum_{\alpha \in \mathbb N^d} f\times \Phi_{\alpha, \alpha}(z) \\
& = &  (2\pi)^{-d/2} \sum_{k=0}^{\infty} f\times \left( \sum_{|\alpha|=k} \Phi_{\alpha, \alpha }(z) \right)\\
  & = & (2\pi)^{-d} \sum_{k=0}^{\infty} f\times \phi_k(z) 
\end{eqnarray*}
where $ \varphi_k$  are the Laguerre functions of type $(d-1)$:
$$ \varphi_k(z)  = (2\pi)^{d/2} \sum_{|\alpha|=k} \Phi_{\alpha, \alpha }(z)=L_k^{d-1}(\frac{1}{2}|z|^2) e^{-\frac{1}{4}|z|^2}.$$
We note that $f\times \Phi_{\alpha, \alpha}$ is an eigenfunction of the operator $L$
with the eigenvalue  ($2|\alpha|+d$). Hence $ (2\pi)^{-d} f\times \phi_k$ is the projection of $f$ onto the eigenspace corresponding  to the eigenvalue $(2k+d).$
The spectral decomposition of the operator $ L $ is given by the special Hermite functions and we can write the same in a compact form as
$$ Lf(z) = (2\pi)^{-d} \sum_{k=0}^\infty (2k+d) f \times \varphi_k(z).$$ 
As in the case of the Hermite operators, one can define and study special Hermite multipliers $ m(L),$  see \cite{ST}. The functions $ \Phi_{\alpha, \beta} $ can be expressed in terms of Laguerre functions. In particular, we have (\cite[Theorem 1.3.5]{ST})
\begin{eqnarray}\label{EPSHF}
 \Phi_{\alpha,0}(z) =  (2\pi)^{-d/2} (\alpha !)^{-1/2} \left(\frac{i}{\sqrt{2}} \right)^{|\alpha|}\bar{z}^{\alpha} e^{-\frac{1}{4} |z|^2}.
\end{eqnarray} 

We also have to deal with the family of operators $ H(\lambda) = -\Delta+\lambda^2 |x|^2$ whose eigenfunctions are given by  scaled Hermite functions.   For $\lambda\in \R^*$ and each $ \alpha \in \mathbb{N}^d $,  we define the family of scaled Hermite functions
$$
\Phi_\alpha^\lambda(x) = |\lambda|^{\frac{d}{4}}\Phi_\alpha(\sqrt{|\lambda|}x), \quad x\in \mathbb{R}^d.
$$
They form an orthonormal basis for $L^2(\mathbb R^d).$ The spectral decomposition of the scaled Hermite operator $ H(\lambda) =  -\Delta+|\lambda|^2 |x|^2 $  is then written as 
\begin{equation}
\label{eq:spectralHermite}
H(\lambda) = \sum_{k=0}^\infty (2k+d)|\lambda| P_k(\lambda),
\end{equation}
for $\lambda \in \mathbb{R}^*$, where $ P_k(\lambda) $ are the (finite-dimensional) orthogonal projections defined on $  L^2(\mathbb{R}^d)$ by
$$
P_k(\lambda)f = \sum_{|\alpha| =k} \langle f,\Phi_\alpha^\lambda \rangle \Phi_\alpha^\lambda,
$$
where $ f \in L^2(\mathbb{R}^d)$. We  can now define scaled special Hermite functions
\begin{eqnarray}
\Phi_{\alpha, \beta}^{\lambda}(z)= (2\pi)^{-d/2}|\lambda|^{d/2} \langle \pi_{\lambda}(z)\Phi_{\alpha}^{\lambda}, \Phi_{\beta}^{\lambda} \rangle.
\end{eqnarray}
They form a complete orthonormal system in $L^2(\mathbb  C^d).$ 
For every $f\in L^{2}(\mathbb C^d),$ we have  the  special Hermite expansion
\begin{eqnarray}
f= \sum_{\alpha \in \mathbb N^d} \sum_{\beta\in \mathbb N^d} \langle f, \Phi_{\alpha, \beta}^{\lambda} \rangle  \Phi_{\alpha, \beta}^{\lambda}.
\end{eqnarray}
We now define the  (scaled) special Hermite operator (twisted Laplacian)  by the relation
\begin{eqnarray}\label{sho}
\mathcal{L}(e^{i\lambda t} f(z)) = e^{i\lambda t} L_{\lambda} f(z). 
\end{eqnarray}
We note that the operators $L_{\lambda}$ and $H(\lambda)$ are related  via the Weyl transform:
$$  \pi_\lambda(L_\lambda f) = \pi_\lambda(f)H(\lambda).$$
In fact, the scaled Hermite functions
are eigenfunctions of the operator $L_{\lambda}:$
\begin{eqnarray}
L_{\lambda} \Phi_{\alpha, \beta}^{\lambda}(z)= (2|\alpha| +d) |\lambda| \Phi_{\alpha, \beta}^{\lambda}(z).
\end{eqnarray}
Note that the eigenvalues of $L_{\lambda}$ are of the form  $(2k+d)|\lambda|, k=0,1, 2,...,$ and the $k^{th}$ eigenspace corresponding to the eigenvalue $(2k+d)|\lambda|$ is infinite-dimensional being the span of $\{ \Phi_{\alpha, \beta}^{\lambda}:|\alpha|=k, \beta \in \mathbb N^d \}.$  It is well known that  the special Hermite  expansion  can be written in the  compact form  as follows:
\begin{eqnarray}
f(z)=(2\pi)^{-d} |\lambda|^d \sum_{k=0}^\infty  f*_\lambda \varphi_k^\lambda(z)
\end{eqnarray}
where $\varphi_{k}^{\lambda}$ is  the scaled Laguerre functions of type $ (d-1)$
\begin{equation}
\label{eq:Laguerre}
\varphi_k^\lambda(z) = L_k^{n-1}\Big(\frac12 |\lambda||z|^2\Big)e^{-\frac14 |\lambda||z|^2}.
\end{equation}
The spectral decomposition of $L_{\lambda}$ is then written as 
$$   L_\lambda f(z) = (2\pi)^{-d} |\lambda|^d \sum_{k=0}^\infty (2k+d)|\lambda|  f*_\lambda \varphi_k^\lambda(z).$$
In particular, recalling $f^{\lambda}$  (see\eqref{eq:inverseFT}), we have
\begin{eqnarray}\label{dso}
L_\lambda f^{\lambda}(z) = (2\pi)^{-d} |\lambda|^d \sum_{k=0}^\infty (2k+d)|\lambda|  f^{\lambda}*_\lambda \varphi_k^\lambda(z).
\end{eqnarray}  
Taking the Fourier transform in  $\lambda$ variable in equation \eqref{dso}, and using \eqref{sho} and  \eqref{eq:lambdaFT}, we obtain the spectral decomposition of the sublaplacian $\mathcal{L}$ on $\mathbb H^d$ as follows:
\begin{equation*}
\label{eq:spectral}
\mathcal{L}f(z,t) = (2\pi)^{-d-1} \int_{-\infty}^\infty \Big( \sum_{k=0}^\infty (2k+d)|\lambda|  f^\lambda*_\lambda \varphi_k^\lambda(z)\Big) e^{-i\lambda t}  |\lambda|^d d\lambda.
\end{equation*}
For more on spectral theory of sublaplacian on Heisenberg  group, we refer to \cite[Chapter 2]{ST3}. 

\subsection{A transference theorem for $ m(\mathcal L)$}\label{tml} Let $ \Gamma_0$ be the subgroup of $ \Hb^d $ consisting of elements of the form $ (0,0, 2\pi k) $ where $ k \in \Z.$ Then $ \Gamma_0$  is a central subgroup and the quotient $ \Hb^d/ \Gamma_0 $ whose underlying manifold is $ \mathbb{R}^d \times \mathbb{R}^d \times \T $ is called the Heisenberg group with compact center.(See \cite[Chapter 4]{ST3})
Let $\mathcal{L}_0$ be the sublaplacian on $\mathbb H^d/\Gamma_0.$ Note that $\mathcal{L}_0f= \mathcal{L}f,$ considering $f$ as a function on $\mathbb H^d.$

Note that on functions $ g $ which are independent of $ t,$ $ \mathcal L_0 g(z) = -\Delta g(z) $ where $ \Delta $ is the Euclidean Laplacian on $ \mathbb C^d.$ The spectral decomposition of the sublaplacian $ \mathcal L_0 $ on this group is given by
$$  \mathcal L_0 f(z,t) = (-\Delta)f^0(z)+ (2\pi)^{-d-1} \sum_{k\in \Z \setminus \{0\}} \Big(\sum_{j=0}^\infty  e^{-ikt} ((2j+d)|k|) f^k\ast_k \varphi_j^k(z)\Big) |k|^d$$  
where $\varphi_j^k(z)= \varphi_j(\sqrt{|k|} z)$ and $f^k$ is defined by the equation 
$f\ast e_j^{k}(z,t)=e^{-ikt} f^k\ast_{k}\varphi_j^k(z),$ $e_{j}^k(z,t)=e^{ikt}\varphi_{j}^{k}(z).$ More generally,  
$$  m(\mathcal L_0) f(z,t) = m(-\Delta)f^0(z) +(2\pi)^{-d-1} \sum_{k\in \Z\setminus \{0\}} \Big(\sum_{j=0}^\infty  e^{-ikt} m((2j+d)|k|) f^k\ast_k \varphi_j^k(z)\Big) |k|^d.$$
There is a transference theorem which connects $ m(\mathcal L) $ on $ \Hb^d $ with $ m(\mathcal L_0) $ on $ \Hb^d/\Gamma_0$ which is the analogue of the classical de Leeuw's theorem (See Theorem \ref{dl}) on the real line.

\begin{theorem}\label{TP} Let $ 1 < p < \infty.$ Suppose $ m(\mathcal L) $ is a bounded operator on $ L^p(\Hb^d).$ Then the transferred operator $ m(\mathcal L_0) $ is bounded on $ L^p(\Hb^d/\Gamma_0).$
\end{theorem}

This theorem has been proved and used in \cite{PKST}. The idea is to realise  $ m(\mathcal L) $  ( $ m(\mathcal L_0) $ )  as an operator valued Fourier multiplier for $ \mathbb{R} $ (resp. $ \T.$). Indeed, writing the Fourier inversion as 
$$  f(z,t) = (2\pi)^{-1} \int_{-\infty}^\infty e^{-i\lambda t} f^\lambda(z) d\lambda $$ and recalling that $ L_\lambda $ is defined by the equation $ \mathcal L(f(z)e^{i\lambda t}) = e^{i\lambda t} L_\lambda f(z) $ we have
$$ m(\mathcal L) f(z,t) =  (2\pi)^{-1} \int_{-\infty}^\infty e^{-i \lambda t} m(L_\lambda) f^\lambda(z) d\lambda $$ 
where $ m(L_\lambda) $ is the multiplier transform associated to the elliptic operator $ L_\lambda $ which has an explicit spectral decomposition. By identifying $ L^p(\Hb^d) $ with $ L^p(\mathbb{R}, L^p(\mathbb C^d)) $ we see that, in view of the above equation, $ m(\mathcal L) $ can be considered as an operator valued Fourier multiplier for the Euclidean Fourier transform on $ \mathbb{R}.$ Similarly, $ m(\mathcal L_0) $ can be thought about as an operator valued multiplier for the Fourier series on $ \T.$ Now, it is an easy matter to imitate the proof of de Leeuw's theorem to prove the above result.

We need the following analogue of Theorem \ref{TP} for the reduced polarised Heisenberg group. Let $\mathcal{\tilde{L}}_0$ be the  sublaplacian on $G^d=\mathbb H^d_{pol}/\Gamma_0.$ We note that $\mathcal{\tilde{L}}_0f= \mathcal{ \tilde{L}}f$ considering $f$ as a function on  $\mathbb H^d_{pol}.$

\begin{theorem}\label{NTP} Let $ 1 < p < \infty.$ Suppose $ m(\mathcal {\tilde {L}}) $ is a bounded operator on $ L^p (\mathbb H^d_{pol}).$ Then the transferred operator $ m(\mathcal{\tilde{L}}_0) $ is bounded on $ L^p(\Hb^d_{pol}/\Gamma_0).$
\end{theorem}

\subsection{Modulation  spaces}\label{MW}
In 1983, Feichtinger \cite{HG4} introduced a  class of Banach spaces,  the so called modulation spaces,  which allow a measurement of space variable and Fourier transform variable of a function or distribution on $\mathbb R^d$ simultaneously using the short-time Fourier transform(STFT). The  STFT  of a function $f$ with respect to a window function $g \in {\mathcal S}(\R^d)$ is defined by
$$  V_{g}f(x,y) = (2\pi)^{-d/2} \int_{\R^d} f(t) \overline{g(t-x)} \, e^{- i y\cdot t} \, dt,  ~  (x, y) \in \mathbb R^{2d} $$
whenever the integral exists. 
For $x, y \in \R^d$ the translation operator $T_x$ and the modulation operator $M_w$ are
defined by $T_{x}f(t)= f(t-x)$ and $M_{w}f(t)= e^{i y\cdot t} f(t).$ In terms of these
operators the STFT may be expressed as
\begin{eqnarray}
\label{ipform} V_{g}f(x,y)=\langle f, M_{y}T_{x}g\rangle\nonumber
\end{eqnarray}
 where $\langle f, g\rangle$ denotes the inner product for $L^2$ functions,
or the action of the tempered distribution $f$ on the Schwartz class function $g$.  Thus $V: (f,g) \to V_g(f)$ extends to a bilinear form on $\mathcal{S}'(\R^d) \times \mathcal{S}(\R^d)$ and $V_g(f)$ defines a uniformly continuous function on $\R^{d} \times \R^d$ whenever $f \in \mathcal{S}'(\R^d) $ and $g \in  \mathcal{S}(\R^d)$.
\begin{definition}[modulation spaces]\label{ms} Let $1 \leq p,q \leq \infty,$ and $0\neq g \in{\mathcal S}(\R^d)$. The  modulation space   $M^{p,q}(\R^d)$
is defined to be the space of all tempered distributions $f$ for which the following  norm is finite:
$$ \|f\|_{M^{p,q}}=  \left(\int_{\R^d}\left(\int_{\R^d} |V_{g}f(x,y)|^{p} dx\right)^{q/p} \, dy\right)^{1/q},$$ for $ 1 \leq p,q <\infty$. If $p$ or $q$ is infinite, $\|f\|_{M^{p,q}}$ is defined by replacing the corresponding integral by the essential supremum. 
\end{definition}

\begin{rem}
\label{equidm}
The definition of the modulation space given above, is independent of the choice of 
the particular window function.  See \cite[Proposition 11.3.2(c), p.233]{gro}. 
\end{rem}

Next, we shall see how the Fourier-Wigner transform  and the STFT are related.
Let $\pi$ be the Schr\"odinger representation of the Heisenberg group  with the parameter $\lambda =1$ which is realized on $L^{2}(\mathbb R^d)$ and explicitly given by 
$$\pi (x,y,t) \phi (\xi)=e^{it} e^{i(x\cdot \xi + \frac{1}{2}x\cdot y)} \phi (\xi +y)$$
where $x,y \in \mathbb R^d, t \in \mathbb R, \phi \in L^{2}(\mathbb R^d).$ The Fourier-Wigner transform of two functions $f, g\in L^{2}(\mathbb R^d)$ is defined by 
$$W_{g}f(x,y)= (2\pi)^{-d/2} \langle \pi (x,y, 0)f, g \rangle.$$
For $z=x+iy,$ we  put  $\pi(x,y, 0)= \pi (z, 0)= \pi(z).$ We may rewrite  the Fourier-Wigner transform as
$$W_{g}f(x,y)=\langle f, \pi^*(z)g \rangle,$$ 
where $\langle f, g\rangle$ denotes the inner product for $L^2$ functions,
or the action of the tempered distribution $f$ on the Schwartz class function $g$.  
We recall the representation  $\rho_1$ (See Section \ref{HP}) of $\mathbb{H}^d_{pol}$, $\rho_{1}=\rho=\pi \circ \Phi^{-1},$  $\rho (x,y, e^{it})$ acting on  $L^{2}(\mathbb R^d)$ is given by 
$$\rho(x,y, e^{it})\phi (\xi)= e^{it} e^{ix\cdot \xi} \phi (\xi +y), \ \  \phi \in L^{2}(\mathbb R^d).$$
We now write the Fourier-Wigner transform  in terms of  the STFT:  Specifically,
we put $\rho(x,y)\phi(\xi)= e^{ix\cdot \xi} \phi (\xi +y),$ and  have 
\begin{eqnarray}\label{ui}
\langle \pi(x,y)f, g \rangle = e^{\frac{i}{2}x\cdot y} \langle \rho(x,y)f, g\rangle= e^{-\frac{i}{2} x\cdot y} V_{g}f(y,-x).
\end{eqnarray}
This useful identity \eqref{ui} reveals that  the definition of modulation spaces we have introduced  in the introduction and in the present section  is essentially the same.

The following basic properties of modulation spaces are well-known   and for the proof we refer the reader  to \cite{gro, HG4}.

\begin{lemma}\label{pl}
\begin{enumerate}
\item The space $M^{p,q}(\mathbb R^{d})  (1\leq p \leq \infty)$  is a  Banach space.

\item  $M^{p,p}(\mathbb R^d)\hookrightarrow L^{p}(\mathbb R^d)$ for $1\leq p \leq 2,$ and $L^{p}(\mathbb R^d) \hookrightarrow M^{p,p}(\mathbb R^d)$ for $2\leq p \leq \infty.$

\item \label{inclu} If $q_{1}\leq q_{2}$ and $p_{1}\leq p_{2}$, then $M^{p_{1}, q_{1}}(\mathbb R^{d}) \hookrightarrow M^{p_{2}, q_{2}}(\mathbb R^{d})$.

\item \label{d} $\mathcal{S}(\mathbb R^d)$ is dense in $M^{p,q}(\mathbb R^d) (1\leq p,q < \infty).$ 
\end{enumerate}
\end{lemma}

\noindent   
We also refer to Gr\"ochenig's book \cite{gro} for the
basic definitions and further properties of modulation spaces.  Finally, we note that there is also an equivalent definition of modulation spaces using frequency-uniform decomposition techniques (which is quite similar in the spirit of Besov spaces), independently studied by Wang et al. in  \cite{WZG}, which has turned out to be very fruitful in PDE, see  \cite{WB}. 

\section{Hermite multipliers via transference theorems}\label{hmvt}

As  mentioned in the introduction we prove our main results via transference techniques.  We first investigate the connection between $m(H) $ acting on $M^{q,p}(\R^d)$ and $ m(\mathcal L_0) $ acting on $ L^p(G^d) $ where $ G^d = \Hb^d_{pol}/\Gamma_0.$ Recall that $ f \in M^{q,p}(\R^d) $ if and only if $ \langle\rho(x,y)f,\Phi_0\rangle \in L^p(\R^d,L^q(\R^d)).$ Moreover, for any left invariant vector filed $ X $ on $ G^d $ we have
$$ X \langle \rho(x,y,t)f,\Phi_0\rangle = \langle \rho(x,y,t)\rho^*(X)f,\Phi_0\rangle $$
where $ \rho^*(X)\varphi = \frac{d}{dt}|_{t=0}\rho(\text{exp}(tX))\varphi.$ A simple calculation shows that $ \rho^*(X_j) = i\xi_{j}  $ and $ \rho^*(Y_j) = \frac{\partial}{\partial \xi_{j}}$ for $ j =1,2,...,d.$ Consequently, we get $ \mathcal L\langle \rho(x,y,t)f,\Phi_0\rangle   = \langle \rho(x,y,t)Hf,\Phi_0\rangle  $ which leads to the identity 
$$ m(\mathcal {\tilde{L}}_0)\langle \rho(\cdot)f,\Phi_0\rangle   = \langle \rho(x,y,t)m(H)f,\Phi_0\rangle $$ via spectral theorem (See \cite[Section 2.3]{ST2} for details). Thus we see that
\begin{eqnarray}\label{hl}
\| m(\mathcal {\tilde{L}}_0)\langle \rho(\cdot)f,\Phi_0\rangle  \|_{L^p(\R^d, L^q(\R^d \times \T))} = \|m(H)f \|_{M^{q,p}(\R^d)}.
\end{eqnarray}
Consequently, the boundedness of $m(H) $ on $ M^{q,p}(\R^d) $ is implied by the boundedness of $ m(\mathcal {\tilde{L}}_0)$ on $L^p(\R^d, L^q(\R^d \times \T)).$ We can now prove Theorem \ref{MT}.\\

\begin{proof}[Proof of Theorem \ref{MT}] We do this in two steps. Assuming that $ m(\mathcal { \tilde{L}}) $ is bounded on $ L^p(\Hb^d_{pol}) $ it follows from Theorem \ref{NTP} that $m(\mathcal {\tilde{L}}_0) $ is bounded on $ L^p(\Hb^d_{pol}/\Gamma).$  As the underlying manifold of $\Hb^d_{pol}/\Gamma $ is $ \R^d \times \R^d \times \T $ we have the boundedness of $ m(\mathcal {\tilde{L}}_0) $ on the space $ L^p(\R^d, L^p(\R^d \times \T)).$ If we can show that under the assumptions on $ p $ and $ q $ stipulated in Theorem \ref{MT}, the operator $m(\mathcal {\tilde{L}}_0) $ is bounded on the mixed norm  space $ L^p(\R^d, L^q(\R^d \times \T))$ then we are done simply by applying $ m(\mathcal {\tilde{L}}_0) $ to $ \langle \pi(z,t)f,\Phi_0\rangle$ (see \eqref{hl}). To this end, we make use of the following result of Herz and Rivi\'ere \cite{HR}.

\begin{theorem}[Herz-Rivi\`ere \cite{HR}]\label{hr}  Let   $G=\Gamma H$  be the semi-direct product of an amenable group $H$  by a locally compact group $\Gamma.$  Assume that a bounded linear operator $T:L^{p}(G) \to L^{p}(G)$  commutes with  left-translations.  Take $q$ such that $1\leq p\leq q  \leq 2$ or   $2\leq q \leq p < \infty.$ Then for any complex-valued continuous function $f$ of compact support on $G,$ we have 

$$\|Tf\|_{L^{p}(\Gamma, L^{q}(H))} \lesssim \| f\|_{L^{p}(\Gamma, L^{q}(H))}.$$ Consequently, $ T $ has a bounded extension to the mixed norm spaces $L^{p}(\Gamma, L^{q}(H)).$

\end{theorem}

In view of the above theorem all we have to do is to realise $ G^d $ as a semidirect product of $ \R^d $ with $ \R^d \times \T.$  As $ m(\mathcal L_0) $ is invariant under left translations, we can then apply the above theorem to arrive at the required conclusion.

Let us briefly recall the definition of a semidirect product for the convenience of the readers.
Let  $H$ be a topological group whose operation  is written as  $+$ and  $\Gamma$  another  topological group,  written as  multiplicatively, such that there is a continuous map $\Gamma \times H \to H $: $(\sigma, x)\mapsto \sigma x$ with $\sigma( x+y) = \sigma x + \sigma y$ and $\tau (\sigma x)= (\tau \sigma) x.$  In this situation, we say that $\Gamma$ acts on $H.$
The semi-direct product $\Gamma H$ is then the   topological space $\Gamma \times H$  with the group operation
$$(\sigma, x)\cdot (\tau, y) = (\sigma \tau, \tau x +y).$$
Let $\Gamma = (\mathbb{R}^d,+)$ be the additive group and $H = (\mathbb{R}^d\times \mathbb T,  \cdot)$ be the  group with following group law:
$$(y,e^{it})\cdot (y',e^{it'}) = (y+y',e^{i(t+t')}).$$
We define a map $\Gamma \times H \to H$ as 
$(x, (y,  e^{it}) ) \mapsto (y, e^{i (t+ xy)}).$ We note  that via this map, $\Gamma$ acts on $H.$
If there is no confusion, we write $(x, (y, e^{it}))= (x, y, e^{it})$ for $(x, (y, e^{it})) \in \Gamma \times H.$
 Forming the   semi-direct product $G=\Gamma H$ we see that the group law is given by $$(x,y,e^{it})(x',y',e^{it'}) = (x+x',y+y',e^{i(t+t'+x'.y)}).$$ This is precisely the group $ G^d = \Hb^d_{pol}/\Gamma$. Hence, by Theorem \ref{hr}  we get the  boundedness  of  $m(\mathcal L_0) $  on the mixed norm  space $ L^p(\R^d, L^q(\R^d \times \T))$. This completes the proof of Theorem \ref{MT}.
 \end{proof}

\begin{proof}[Proof of Theorem \ref{FT}]
Using Theorems \ref{MSH} and \ref{NTP}, we conclude that $m(\mathcal{ \tilde{L}}_0)$ is bounded on $L^{p}(\mathbb H^d_{pol}).$  We now apply Theorem \ref{MT}, to complete the proof.
\end{proof}

\section{Hermite multipliers via Fourier multipliers on torus}\label{hmft}
In the last section, we have proved  the boundedness of Hermite multiplier on modulation spaces via transference results. Specifically, we have proved that the  boundedness of multiplier operators  on  Heisenberg groups ensures the boundeness of  corresponding Hermite multiplier operators on modulation spaces.

In this section, we  first we prove the transference result which connects  Hermite multipliers on modulation spaces and Fourier multipliers on  $L^p(\mathbb T^d).$  Specifically,  our result (Proposition \ref{tpt})  states that  boundedness of Fourier multipliers on torus  guarantees  the boundedness of Hermite multipliers on modulation spaces.   In order to find a  fruitful application of our result, we need to know the boundedness of Fourier multiplier on $L^{p}(\mathbb T^d)$.  To this end, we  prove(Proposition \ref{lp}) boundedness of Fourier multiplier on $L^{p}(\mathbb R^d)-$and then use the celebrated theorem of  de Leeuw to come back to the Fourier multiplier on  $L^{p}(\mathbb T^d).$  Combining these results,  finally in this  section we prove  Theorem \ref{HMT}.
\noindent
For  the sake of convenience of the reader, we recall definitions:
\begin{definition}[Hermite multipliers on $M^{p,q}(\mathbb R^d)$]\label{HM} Let $m$ be a bounded function on $\mathbb N^d.$ We say that $m$ is a Hermite multiplier on the space $M^{p,q}(\mathbb R^d)$ if the linear operator defined by
$$T_{m}f =\sum_{\alpha \in \mathbb N^d} m(\alpha) \langle f, \Phi_{\alpha}  \rangle  \Phi_{\alpha},   \ (f\in \mathcal{S}(\mathbb R^d))$$
extends to a bounded linear operator from $M^{p,q}(\mathbb R^d)$ into itself, that is, $\|T_{m}f\|_{M^{p,q}} \lesssim \|f\|_{M^{p,q}}.$
\end{definition}
\begin{definition}[Fourier multiplier on $L^{p}(\mathbb T^d)$]  Let $m$ be a bounded measurable function  defined on  $\mathbb Z^d.$ We say that  $m$ is a Fourier multiplier on $L^{p}(\mathbb T^d)$ if the linear operator $T_m$ defined by
$$\widehat{(T_mf)}(\alpha)=m(\alpha)\hat{f}(\alpha), \   (f\in \mathcal{P}(\mathbb T^d), \alpha \in \mathbb Z^d)$$
where  $\hat{f}(\alpha)= \int_{\mathbb T^d} f(\theta) e^{- i \theta \cdot \alpha} d\theta$  are the Fourier coefficients of $ f$, extends to a bounded linear operator from $L^{p}(\mathbb T^d)$ into itself, that is, $\|T_{m}f\|_{L^{p}(\mathbb T^d)} \lesssim \|f\|_{L^{p}(\mathbb T^d)}.$
  \end{definition}

Now we prove a  transference result for Fourier multiplier on torus and  Hermite multiplier on modulation spaces$-$which is of interest in itself.  Specifically, we  have the following proposition.

\begin{proposition}\label{tpt} Let $1\leq p < \infty.$
If $m:\mathbb Z^d \to \mathbb C $ is a Fourier multiplier on $L^p(\mathbb{T}^d)$, then $m|_{\mathbb N^d}$,  the restriction of $m$ to $\mathbb N^d,$  is  a Hermite multiplier on $M^{p,p}(\mathbb R^d)$.
\end{proposition}
\begin{proof} Since $\mathcal{S}(\mathbb R^d)$ is dense (Lemma \ref{pl}\eqref{d}) in $M^{p,p}(\mathbb R^d),$ it would be sufficient to prove
\begin{eqnarray}\label{tmc}
\|T_mf\|_{M^{p,p}} \lesssim \|f\|_{M^{p,p}}
\end{eqnarray}
for $f\in \mathcal{S}(\mathbb R^d).$ In fact,  by density argument, it follows that $T_m$ has  a bounded extension  to   $M^{p,p}(\mathbb R^d),$ that is,  inequality \eqref{tmc} holds true for $f\in M^{p,p}(\mathbb R^d).$ 

Let $f\in \mathcal{S}(\mathbb R^d).$ Then using special Hermite functions \eqref{SHF} and their property \eqref{EPSHF}, we  compute the modulation space norm (See Section \ref{MW} ) of Hermite multiplier operator (See Definition \ref{HM}), and we obtain
\begin{eqnarray}\label{is}
\|T_mf\|^p_{M^{p,p}} & =  & (2\pi)^{-dp/2} \int_{\mathbb R^{2d} } | \langle \pi(x,y)T_mf,\Phi_0\rangle |^p dy  \ dx \nonumber \\
& = & (2\pi)^{-dp/2}  \int_{\mathbb R^{2d}} \left|\sum_ {\alpha \in \mathbb N^d}m(\alpha)\langle f,\Phi_\alpha\rangle\langle \pi(x,y)\Phi_\alpha,\Phi_0\rangle\right|^p dy \ dx \nonumber \\
& = &  \int_{\mathbb C^d} \left|\sum_{\alpha \in \mathbb N^d } m(\alpha)\langle f,\Phi_\alpha\rangle\Phi_{\alpha,0}(z)\right|^p dz \nonumber \\
& = & (2\pi)^{-dp/2}\int_{\mathbb C^d} \left|\sum_{\alpha \in \mathbb N^d } m(\alpha)\langle f,\Phi_\alpha\rangle\frac{i^{|\alpha|} \bar{z}^\alpha}{\sqrt{\alpha !}\,2^{|\alpha|/2}} e^{-\frac{1}{4}|z|^2}\right|^p dz.
\end{eqnarray}
By using polar coordinates  $z_j= r_j e^{i \theta_j}$,  $r_j:=|z_j|\in [0, \infty), z_j\in \mathbb C$ and  $\theta_j \in [0, 2\pi),$ we get
\begin{eqnarray}\label{LPC}
z^{\alpha}= r^{\alpha} e^{i \alpha \cdot \theta} \ \  \text{and} \ \  dz= r_1 r_2 \cdots r_d d\theta dr
\end{eqnarray}
 where $r = (r_1,\cdots, r_d)$, $\theta = (\theta_1,\cdots,\theta_d)$, $dr = dr_1\cdots dr_d$, $d\theta = d\theta_1\cdots d\theta_d, |r|=\sqrt{\sum_{j=1}^d r_j^2}.$ \\ 
By writing the integral over $\mathbb C^d= \mathbb R^{2d}$ in polar coordinates in
each time-frequency pair and using  \eqref{LPC}, we have
\begin{eqnarray}\label{apc}
\int_{\mathbb C^d} \left|\sum_{\alpha\in \mathbb N^d} m(\alpha)\langle f,\Phi_\alpha\rangle\frac{i^{|\alpha|} \bar{z}^\alpha}{\sqrt{\alpha !}\,2^{|\alpha|/2}} e^{-\frac{1}{4}|z|^2}\right|^p dz\\  = 
\prod_{j=1}^d  \int_{\mathbb{R}^+}\int_{[0,2\pi]}\left|\sum_{\alpha \in \mathbb N^d } m(\alpha)\left(\langle f,\Phi_{\alpha}\rangle\frac{i^{|\alpha|} r^{\alpha}}{\sqrt{\alpha !}\,2^{|\alpha |/2}} e^{-\frac{1}{4}|r|^2}\right)e^{-i\alpha\cdot \theta}\right|^p\,r_j\,d\theta_j\,dr_j. \nonumber
\end{eqnarray}

We put  $a_\alpha= \langle f,\Phi_\alpha\rangle\frac{i^{|\alpha|} r^\alpha}{\sqrt{\alpha !}\,2^{|\alpha|/2}} e^{-\frac{1}{4}|r|^2}.$  Since $|\langle f, \Phi_{\alpha} \rangle | \leq \|f\|_{L^2} \|\Phi_{\alpha} \|_{L^2} \leq \|f\|_2$ ,   the series  $\sum_{\alpha \in \mathbb N^d} |a_\alpha| $ converges. Thus there exists a continuous function $g\in L^{p}(\mathbb T^d)$ with  Fourier coefficients $\hat{g}(\alpha)= a_\alpha$ for $\alpha \in \mathbb N^d$ and $\hat{g}(\alpha)=0$ for $\alpha \in \mathbb Z^{d}\setminus \mathbb N^d.$  In fact, the Fourier series of this $g$ is absolutely convergent, and therefore we may  write 
\begin{eqnarray}\label{FR}
g(e^{i\theta}) = \sum_{\alpha \in \mathbb Z^d} a_{\alpha} e^{-i \alpha \cdot \theta}.
\end{eqnarray}
 Since $m$ is a Fourier  multiplier on $L^{p}(\mathbb T^d),$  \eqref{FR} gives
\begin{eqnarray}\label{fs}
\int_{[0, 2\pi]^d} \left|\sum_{\alpha \in \mathbb N^d} m(\alpha) a_{\alpha} e^{i\alpha \cdot \theta}\right |^p d\theta \lesssim \int_{[0, 2\pi]^d} \left|\sum_{\alpha\in \mathbb N^d} a_{\alpha} e^{-i \alpha \cdot \theta}\right|^p d\theta.
\end{eqnarray}
Now taking \eqref{is}, \eqref{apc}, and \eqref{fs} into  account, we obtain
\begin{align*}
\|T_mf\|_{M^{p,p}} &\lesssim (2\pi)^{-d/2} \left(  \prod_{j=1}^d  \int_{\mathbb{R}^+}\int_{[0,2\pi]}\left|\sum_{\alpha \in \mathbb N^d}\left(\langle f,\Phi_{\alpha}\rangle\frac{i^{|\alpha |} r^{\alpha}}{\sqrt{\alpha !}\,2^{|\alpha|/2}} e^{-\frac{1}{4}|r|^2}\right)e^{-i\alpha \cdot  \theta}\right|^p\,r_j\,d\theta_j\,dr_j\right)^{\frac{1}{p}}\\
& = (2\pi)^{-d/2} \left(\int_{\mathbb C^d} \left|\sum_{\alpha \in \mathbb N^d } \langle f,\Phi_\alpha\rangle\frac{i^{|\alpha|} \bar{z}^\alpha}{\sqrt{\alpha !}\,2^{|\alpha|/2}} e^{-\frac{1}{4}|z|^2}\right|^p dz\right)^{\frac{1}{p}} \\
& =  \left(\int_{\mathbb C^d} \left|\sum_{\alpha \in \mathbb N^d } \langle f,\Phi_\alpha\rangle\Phi_{\alpha,0}(z)\right|^p dz \right)^{\frac{1}{p}} \\
& =  (2\pi)^{-d/2} \left(\int_{\mathbb R^{2d}} \left|\sum_ {\alpha \in \mathbb N^d}\langle f,\Phi_\alpha\rangle\langle \pi(x,y)\Phi_\alpha,\Phi_0\rangle\right|^p dy \ dx\right)^{\frac{1}{p}}\\
&=  \|f\|_{M^{p,p}}.
\end{align*} 
Thus, we conclude that $\|T_mf\|_{M^{p,p}} \lesssim \|f\|_{M^{p,p}}$ for $f\in \mathcal{S}(\mathbb R^d).$ 
This completes the proof of Proposition \ref{tpt}. 
\end{proof}
Next we recall the celebrated  theorem of de Leeuw, which  gives the relation between  Fourier multipliers on Euclidean  spaces and tori.  To do this, we start with
\begin{definition}[Fourier multiplier on $L^{p}(\mathbb R^d)$]\label{fme} Let $m$ be a bounded measurable function  defined on  $\mathbb R^d.$ We say that  $m$ is a Fourier multiplier on $L^{p}(\mathbb R^d)$ if the linear operator $T_m$ defined by
$$\widehat{(T_mf)}(\xi)=m(\xi)\hat{f}(\xi), \   (f\in \mathcal{S}(\mathbb R^d), \ \xi\in \mathbb R^d)$$
where $\hat{f}$ is the Fourier transform, extends to a bounded linear operator from $L^{p}(\mathbb R^d)$ into itself, that is, $\|T_{m}f\|_{L^{p}} \lesssim \|f\|_{L^{p}}.$
  \end{definition}
\begin{theorem}[de Leeuw]\label{dl} If  $m:\mathbb R^d\to \mathbb C$ is continuous and  Fourier multiplier on $ L^{p}(\mathbb R^d)  (1\leq p < \infty)$,  then  $m|_{\mathbb Z^d},$ the restriction of $m$ to $\mathbb Z^d$,  is a Fourier multiplier on $L^{p}(\mathbb T^d).$
\end{theorem}
\noindent 
In order prove Theorem \ref{HMT}, we also need the boundedness of  a following Fourier multiplier 
\begin{eqnarray}\label{ncm}
m(\xi)= \frac{e^{i (2|\xi|_1+d)^\gamma}}{(2|\xi|_1+d)^\beta} \ \ \  (\beta> 0, \gamma>0, \xi \in \mathbb R^d)
\end{eqnarray}
where $|\xi|_1 := \sum_{j=1}^d |\xi_j|,$ on $L^{p}(\mathbb R^d).$  This we shall prove in the next proposition.   We note that in  \cite[Theorem 1]{tps}, it is  proved that   the function $m(\xi)$ which  is 0 near the origin and  $|\xi|^{-\beta} e^{i|\xi|^{\alpha}}  \   \   (|\xi| = \sqrt{ \sum_{j=1}^d} \xi_j^2 )$   outside a compact set  is  a Fourier multiplier on  $L^{p}(\mathbb R^d)$ for a suitable choice of $\alpha, \beta, p, d$. Our proof of Proposition \ref{lp} is motivated by this result.

\begin{proposition}\label{lp} Let $m$ be  given by \eqref{ncm}.  Then $m$ is a Fourier multiplier on $L^{p}(\mathbb R^d)$ for $|\frac{1}{p}-\frac{1}{2}|<\frac{\beta}{d\gamma}$.
\end{proposition}
To prove this proposition we need the following technical lemma.  We will prove this lemma at the end.
\begin{lemma} \label{tl} Let $\sigma>0,  \gamma>0$ and
\begin{eqnarray}\label{k}
k_\sigma(x) = \frac{1}{(2\pi)^{d/2}}\int_{\mathbb{R}^d} e^{i(2|\xi|_1+d)^\gamma} e^{-\sigma (2|\xi|_1+d)^\gamma} e^{ix \cdot \xi}\,d\xi \ (x\in \mathbb R^d).
\end{eqnarray}
Then
 \begin{equation*}
\|k_\sigma\|_{L^1}\lesssim \sigma^{-d/2}e^{-\frac{1}{2}\sigma d^\gamma}.
\end{equation*}
\end{lemma}
\begin{proof}[Proof of Proposition \ref{lp}] 
Performing a simple change of variables in the gamma function, we write
\begin{eqnarray}\label{gf}
(2|\xi|_1+d)^{-\beta} = \frac{1}{\Gamma(\beta/\gamma)}\int_0^\infty \sigma^{\frac{\beta}{\gamma}-1}\exp(-\sigma(2|\xi|_1+d)^\gamma)\,d\sigma.
\end{eqnarray}
In view of  Definition \ref{fme}, \eqref{ncm} and \eqref{gf}, we write
\begin{eqnarray}\label{ms}
T_mf(x) & = & (2\pi)^{-d/2} \int_{\mathbb R^d} m(\xi) \hat{f}(\xi) e^{i x\cdot \xi} d\xi \nonumber \\
& = &  \frac{1}{\Gamma(\beta/\gamma)}\int_0^\infty \sigma^{(\beta/\gamma)-1}(k_\sigma\ast f)(x)\,d\sigma
\end{eqnarray}
where $k_\sigma$ is as defined in Lemma \ref{tl}. 
Using  \eqref{ms}, we have \begin{equation}\label{eq8}
\|T_mf\|_{L^p}\leqslant \frac{1}{\Gamma(\beta/\gamma)}\int_0^\infty \sigma^{(\beta/\gamma)-1}\|k_\sigma\ast f\|_{L^p}\,d\sigma.\end{equation}
The boundedness of the multiplier operator  $T_mf$ on $L^p$ will follow if we could show that the operator $ f \rightarrow k_\sigma \ast f$ is bounded on $ L^p.$ We shall achieve this by using  Riesz-Thorin interpolation theorem  and the  standard  duality argument. To this end, we use Lemma \ref{tl} and Young's inequality, to obtain
 \begin{equation}\label{b1}\|k_\sigma\ast f\|_{L^1}\lesssim \sigma^{-d/2}e^{-\frac{1}{2}\sigma d^\gamma}\|f\|_{L^1}.\end{equation}
 Since  $|\hat{k}_\sigma(\xi)|\leqslant e^{-\frac{1}{2}\sigma d^\gamma}$,  Plancherel theorem gives
\begin{equation}\label{b2}   
 \|k_\sigma\ast f\|_{L^2}\leqslant e^{-\frac{1}{2}\sigma d^\gamma}\|f\|_{L^2}. 
\end{equation}
Taking  \eqref{b1}, and \eqref{b2} into  account,  Riesz-Thorin interpolation theorem gives \begin{equation}\label{eq7}\|k_\sigma\ast f\|_{L^p}\leqslant C_1 \sigma^{-\lambda} e^{-\frac{1}{2} \sigma d^\gamma}\|f\|_{L^p}  \end{equation} where $\lambda = d\left(\frac{1}{p}-\frac{1}{2}\right)$, for $1\leq p \leq 2.$
Finally using \eqref{eq7} and \eqref{eq8}, we see that $T_m$ is bounded from $L^p(\mathbb{R}^d)\rightarrow L^p(\mathbb{R}^d)$ if \[\int_0^\infty \sigma^{\beta/\gamma-1-\lambda}e^{-\frac{1}{2}\sigma d^\gamma}\,d\sigma <\infty,\] which happens if and only if $\frac{1}{p}-\frac{1}{2}<\frac{\beta}{d\gamma}$. This proves the theorem for the case when $1<p<2$ and the case $p\geqslant 2$ follows from the duality.
\end{proof}
Now we shall prove our  Lemma \ref{tl}.
\begin{proof}[Proof of Lemma \ref{tl}] Since  $k_{\sigma}$  (see \eqref{k}) is the inverse Fourier transform of the function $\exp{ ( (i-\sigma) (2|\xi|_1+d)^{\gamma})}$, we have 
\begin{eqnarray}\label{fk}
\widehat{k_{\sigma}}(\xi)=e^{ (i-\sigma) (2|\xi|_1+d)^{\gamma}}.
\end{eqnarray} 
Let $\alpha = (\alpha_1,\cdots,\alpha_d) \in \mathbb N^d$ be a multi-index of length $l$, that is, $\sum_{j=1}^d \alpha_j=l$, and put $D^{\alpha}_{\xi}= \left(\frac{\partial}{\partial \xi_1} \right)^{\alpha_1}\cdots \left(\frac{\partial}{\partial \xi_d} \right)^{\alpha_d}, \  \  \xi= (\xi_1,..., \xi_d)\in \mathbb R^d.$ Then in view of \eqref{fk}, we have \begin{equation}\label{eq2}
|D^\alpha_\xi\widehat{k_\sigma}(\xi)|\lesssim (2|\xi|_1+d)^{-l(1-\gamma)}e^{-\frac{1}{2}\sigma(2|\xi|_1+d)^\gamma}e^{-\frac{1}{2}\sigma d^\gamma}.
\end{equation} 
To obtain inequality \eqref{eq2}, we  have used these ideas:  After taking partial derivatives of higher order for $\widehat{k_{\sigma}}$, we have estimated all the powers of $(2|\xi|_1+d)$ by the highest appearing power$-$which is $(2|\xi|_1+d)^{-l(1-\gamma)}$, this we could do because $(2|\xi|_1+d)\geqslant 1$.  We have also  dominated $e^{-\frac{1}{2}\sigma(2|\xi|_1+d)^\gamma}$ by $e^{-\frac{1}{2}\sigma d^\gamma}$, which is obvious.\\
Next, by Plancherel's theorem and \eqref{eq2}, we get\[\||x|^lk_\sigma\|_{L^2}\lesssim \left(\int_{\mathbb{R}^d}(2|\xi|_1+d)^{-2l(1-\gamma)}e^{-\sigma (2|\xi|_1+d)^\gamma}\,d\xi\right)^{\frac{1}{2}}e^{-\frac{1}{2}\sigma d^\gamma}.\] Performing  a change of variable,  we get \[\||x|^lk_\sigma\|_{L^2}\lesssim \sigma^{-\frac{d}{2\gamma}+\frac{l(1-\gamma)}{\gamma}}e^{-\frac{1}{2}\sigma d^\gamma}\left(\int_{\mathbb{R}^d}g(2|\xi|_1+\sigma^{1/\gamma}d)\,d\xi\right)^{1/2},\] where $g(t) = t^{-2l(1-\gamma)}e^{-t^\gamma} \ (t>0)$. Noting that $\int g(2|\xi|_1+\sigma^{1/\gamma}d)\,d\xi<\infty$ for all $\gamma$ and $l$, we conclude \[\||x|^lk_\sigma\|_{L^2}\lesssim\sigma^{-\frac{d}{2\gamma}+\frac{l(1-\gamma)}{\gamma}}e^{-\frac{1}{2}\sigma d^\gamma}.\] Next we use the fact that for $h\in L^2(\mathbb{R}^d)\cap L^1(\mathbb{R}^d)$, $l>d/2$ and $R>0$, we have \[\|h\|_{L^1}\lesssim  R^{d/2}(\|h\|_{L^2}+R^{-l}\||x|^l h\|_{L^2})\] for all $R>0$. Taking $h = k_\sigma$ and $R = \sigma^{\frac{1-\gamma}{\gamma}}$, we obtain
\[\|k_\sigma\|_{L^1}\leqslant C\sigma^{-d/2}e^{-\frac{1}{2}\sigma d^\gamma}.\]  This completes the proof.
\end{proof}
\noindent
We now use Proposition \ref{lp} and Theorem  \ref{dl} to obtain the following corollary.
\begin{corollary} \label{lc} Let $\left |\frac{1}{p}- \frac{1}{2}\right |< \frac{\beta}{d \,\gamma },  \  \beta>0, \gamma>0, \  |\alpha|_1= \sum_{j=1}^d|\alpha_j|$.  Then the  sequence  $\{(2|\alpha|_1+d)^{-\beta}\exp(i(2|\alpha|_1+d)^\gamma)\}_{\alpha\in\mathbb{Z}^d}$ defines  a   multiplier on $L^p(\mathbb{T}^d)$.
\end{corollary} 
\begin{proof}[Proof of Theorem \ref{HMT}] Using Corollary \ref{lc} and Proposition \ref{lp}, we  may deduce that $m(H)$ is bounded on $M^{p,p}(\mathbb R^d).$ This completes the proof.
\end{proof}

\section{Hermite multiplier for Schr\"odinger propagator}\label{hmsos}
In this section, we prove the boundedness of Schr\"odinger propagator $m(H)=e^{itH}$  using the 
properties of Hermite and special Hermite functions. Our approach of proof illustrates  how these functions nicely fit into  modulation spaces$-$and prove useful estimate. 

\begin{proof}[Proof of Theorem \ref{mso}] Let $f\in \mathcal{S}(\mathbb R^d).$
Then we have the  Hermite expansion of  $f$ as follows:
\begin{eqnarray}\label{he}
f= \sum_{\alpha \in \mathbb N^d} \langle f, \Phi_{\alpha}\rangle \Phi_{\alpha}.
\end{eqnarray} 
Now using \eqref{he} and \eqref{SHF}, we obtain
\begin{eqnarray}\label{pihe}
\langle \pi(z)f, \Phi_0 \rangle & =  & \sum_{\alpha \in \mathbb N^d}\langle f, \Phi_{\alpha} \rangle \langle \pi(z)\Phi_{\alpha}, \Phi_{0} \rangle \nonumber\\
& = & \sum_{\alpha \in \mathbb N^d} \langle f, \Phi_{\alpha} \rangle \Phi_{\alpha, 0}(z).
\end{eqnarray}
Since $\{\Phi_{\alpha} \}$ forms an orthonormal basis for $L^{2}(\mathbb R^d),$ \eqref{pihe} gives
\begin{eqnarray}\label{ff}
\langle \pi(z)m(H)f, \Phi_0  \rangle  & = &  \sum_{\alpha \in \mathbb N^d} \langle m(H)f, \Phi_{\alpha} \rangle \Phi_{\alpha, 0}(z) \nonumber \\
& = & \sum_{\alpha \in \mathbb N^d} m(2|\alpha| +d) \langle f, \Phi_{\alpha} \rangle \Phi_{\alpha, 0}(z).\nonumber
\end{eqnarray}
Therefore, for $m(H) = e^{itH}$, we have \begin{align}\label{1}
\langle \pi(z)e^{itH}f, \Phi_0  \rangle  &= e^{itd} \sum_{\alpha \in \mathbb N^d} e^{2it|\alpha|} \langle f, \Phi_{\alpha} \rangle \Phi_{\alpha, 0}(z) \nonumber \\
&= e^{itd} (2\pi)^{-d/2}  \sum_{\alpha \in \mathbb N^d} e^{2it|\alpha|} \langle f, \Phi_{\alpha} \rangle  (\alpha !)^{-1/2} \left(\frac{i}{\sqrt{2}} \right)^{|\alpha|}\bar{z}^{\alpha} e^{-\frac{1}{4} |z|^2}.
\end{align}
In view of \eqref{ui} and \eqref{1}, we have 
\begin{align}\label{o}
\|e^{itH}f\|^p_{M^{p,p}}&=  \frac{1}{(2\pi)^{d/2}}\int_{\mathbb{C}^d}  \left |  \sum_{\alpha \in \mathbb N^d} e^{2it|\alpha|} \langle f, \Phi_{\alpha} \rangle  (\alpha !)^{-1/2} \left(\frac{i}{\sqrt{2}} \right)^{|\alpha|}\bar{z}^{\alpha} e^{-\frac{1}{4} |z|^2}\, \right |^p dz.
\end{align} 
\noindent
Using  polar coordinates as above (see \eqref{LPC}), we have
\begin{eqnarray}\label{h} \int_{\mathbb{C}^d}  \left |  \sum_{\alpha \in \mathbb N^d} e^{2it|\alpha|} \langle f, \Phi_{\alpha} \rangle  (\alpha !)^{-1/2} \left(\frac{i}{\sqrt{2}} \right)^{|\alpha|}\bar{z}^{\alpha} e^{-\frac{1}{4} |z|^2}\, \right |^p dz&\\
=\prod_{j=1}^d
 \int_{\mathbb{R}^+}\int_{[0,2\pi]} \left|\sum_{\alpha  \in \mathbb N^d } \langle f, \Phi_{\alpha} \rangle  (\alpha !)^{-1/2} \left(\frac{i}{\sqrt{2}} \right)^{|\alpha |}  r^{\alpha}  e^{i \sum_{j=1}^d(2t-\theta_j) \alpha_j} e^{-\frac{1}{4} |r|^2} \right |^p r_j  dr_j d\theta_j.  \nonumber 
\end{eqnarray}
By a simple change of variable $(\theta_j-2t)\rightarrow \theta_j$, we obtain
\begin{eqnarray}\label{b}  \prod_{j=1}^d
 \int_{\mathbb{R}^+}\int_{[0,2\pi]} \left|\sum_{\alpha \in \mathbb N^d } \langle f, \Phi_{\alpha} \rangle  (\alpha !)^{-1/2} \left(\frac{i}{\sqrt{2}} \right)^{ |\alpha |}  r^{\alpha} e^{i \sum_{j=1}^d(2t-\theta_j) \alpha_j} e^{-\frac{1}{4} |r|^2} \right |^p r_j  dr_j d\theta_j  \\
 = \prod_{j=1}^d \int_{\mathbb{R}^+}\int_{[0,2\pi]} \left|\sum_{\alpha \in \mathbb N^d } \langle f, \Phi_{\alpha} \rangle  (\alpha !)^{-1/2} \left(\frac{i}{\sqrt{2}} \right)^{|\alpha |}  r^{\alpha} e^{-i \theta \cdot \alpha} e^{-\frac{1}{4} |r|^2} \right |^p r_j  dr_j d\theta_j . \nonumber
\end{eqnarray}
Combining  \eqref{o}, \eqref{h},  \eqref{b}, and Lemma \ref{pl}\eqref{d},  we have $\|e^{itH}f \|_{M^{p,p}} = \|f\|_{M^{p,p}}$ for $f\in M^{p,p}(\mathbb R^d).$ This completes the proof of Theorem \ref{mso}.
\end{proof}
\noindent
\textbf{Acknowledgment}.  The work leading to this article began  while DGB  was a project  assistant with Professor  Thangavelu in IISc.  DGB is very grateful to  Professor Thangavelu for providing the funding  and arranging research facilities  during his stay at IISc.  DGB is also thankful to DST-INSPIRE and TIFR CAM for the current support. RB  wishes to thank UGC-CSIR for financial support.  ST is supported by J C Bose National Fellowship from D.S.T., Govt. of India.

\end{document}